\documentclass[12pt]{article}
\usepackage[cp1251]{inputenc}
\newcommand{\CopyName}{O.\ D.\ Trofymenko} 
\newcommand{\Year}{2013} 
\newcommand{\Volume}{?} 
\newcommand{\Number}{?} 
\newcommand{\Pages}{?--?} 
\usepackage[english,russian,ukrainian]{babel}
\usepackage{ms_we2} \renewcommand{\logo}{}
\date{}

\begin{document}
\renewcommand{\refname}{\refnam}

\title{Two-radii theorem for solutions of some mean value equations}
\author{O.D. Trofymenko}

\maketitle

\begin{abstract}
 \noindent{\bf Abstract.} A description of solutions of some integral equations has been obtained.
 A two-radii theorem is obtained as well.
\end{abstract}



\section{Introduction}

Characterization of solutions for differential equations in terms of various integral mean values has been studied by many authors (see \cite{1.} - \cite{9.} and references in these papers).

The classes of functions on subsets of the compact plane that satisfy the conditions of the next type is studied in this work

\begin{equation}
\overset{m-1}{\underset{n=s}{\sum}}\frac{r^{2n+2}}{2(n-s)!(n+1)!}{\left(\frac{\partial}{\partial{z}}\right)}^{n-s}{\left(\frac{\partial}{\partial{\bar{z}}}\right)}^nf(z)=\frac{1}{2\pi}\underset{|\zeta-z|\leq{r}}{{\int}{\int}}f(\zeta)(\zeta-z)^sd\xi
d\eta,
\end{equation}
 where $m\in{\mathbb{N}}$ and $s\in {0, ..., m-1}$ are fixed. Also $r$ is fixed or belongs to the set of two elements.

We point out that this equation holds for $m$-analytic functions (see \cite{10.}).
Function from  $C^{2m-2-s}$ in some domain, that satisfies (1) with all possible $z$ and  $r$ is of great interest.

The main results of this work are as follow. \\
1) The description of all smooth solutions for (1) in a disk with radius $R>r$ with one fixed $r$ is obtained (see Theorem 1 below); \\
2) The two-radii theorem is obtained. It turn out that this theorem characterizes class of solution for equation
 \begin{equation}
 {\left(\frac{\partial}{\partial{z}}\right)}^{m-s}{\left(\frac{\partial}{\partial{\bar{z}}}\right)}^mf=0
 \end{equation}
 in terms of equation (1) (see Theorem 2 ). \\
Note that the case $s\geq m$ that corresponds to the zero integral mean value in the right hand side of (1), has been studied in the works of L.Zalcman and V.V.Volchkov  (see \cite{3.}, \cite{11.} - \cite{12.}). The first results that deal with the mean value theorem for polyanalytic functions, are contained in \cite{13.} - \cite{14.}.

\section{Main results}
 Let $J_\nu$ be the Bessel function of the first kind with index $\nu$. For $\rho\geq 0$, $\lambda\in{\mathbb{C}}$, $k\in{\mathbb{Z}}$, let
\[
\Phi_{\lambda,\eta,k}(\rho)=\left(\frac{d}{dz}\right)^\eta\left(J_k(z\rho)\right)|_{z=\lambda}.
 \]
Let also
\[
g_r(z)=\frac{J_{s+1}(rz)}{(zr)^{s+1}}-\overset{m-1}{\underset{n=s}{\sum}}\frac{(zr)^{2(n-s)}(-1)^{n-s}}{(n+1)!(n-s)!2^{2n-s+1}},
\]
 and $Z(g_r)=\left\{z\in{\mathbb{C}}: g_r(z)=0\right\}$, \\
 $Z_r=Z(g_r) \setminus \left( \left\{z\in{\mathbb{C}}:\Re z>0 \right\}\cup \left\{z\in{\mathbb{C}}:\Im z\geq0, \Re z=0 \right\}\right)$.
For $\lambda\in{Z_r}$ by the symbol $n_\lambda$ we denote the multiplicity of zero $\lambda$ of the entire function $g_r$.\\
 Let $\mathbb{D}_R=\left\{z\in{\mathbb{C}}: |z|<R\right\}$. To any function $f\in C(\mathbb{D}_R)$ there  the corresponds Fourier series
\begin{equation}
f(z)\sim\overset{\infty}{\underset{k=-\infty}{\sum}}f_{k}(\rho)e^{ik\varphi},
\end{equation}
where
\begin{equation}
f_{k}(\rho)=\frac{1}{2\pi}\overset{\pi}{\underset{-\pi}{\int}}f(\rho
e^{it})e^{-ikt}dt
\end{equation}
and $0\leq\rho<R$. \\
The next result gives a description for all solutions (1) in a class $C^{\infty}(\mathbb{D}_R)$ with one fixed $r<R$.

\begin{theorem}

 Let $r>0$, $m\in{\mathbb{N}}$ and $s\in{0, ..., m-1}$ are fixed. Let also $R>r$ and a function $f$ belongs to $C^{\infty}(\mathbb{D}_R)$. Then the next statements are equivalent. \\
1) With $|z|<R-r$ equality (1) holds. \\
2) For any $k\in{\mathbb{Z}}$ on $[0,R)$ the next equality holds
\begin{equation}
f_k(\rho)=\underset{p+k\geq0}{\underset{0\leq p\leq s-1}{\sum}}a_{k,p}\rho^{2p+k}+\overset{m-s-1}{\underset{p=0}{\sum}}b_{k,p}\rho^{2p+s+|k+s|}+\underset{\lambda\in{Z_r}}{\sum}\overset{n_\lambda-1}{\underset{\eta=0}{\sum}}
c_{\lambda,\eta,k}\Phi_{\lambda,\eta,k}(\rho)
\end{equation}
where $a_{k,p}\in\mathbb{C}$, $b_{k,p}\in\mathbb{C}$, $c_{\lambda,\eta,k}\in\mathbb{C}$ and
\begin{equation}
c_{\lambda,\eta,k}=O(|\lambda|^{-\alpha})
\end{equation}
 as $\lambda\rightarrow\infty$ for any fixed $\alpha>0$.
 \end{theorem}

Note that analogues of the Theorem 1 for other equations related to ball mean values, were obtained by V.V.Volchkov for the first time (see \cite{5.} - \cite{6.} and the references in these papers).\\
Then let $Z(r_1,r_2)=Z_{r_1} \cap Z_{r_2}$. \\
We formulate now the local two-radii theorem for equation (1).
\begin{theorem}
 Let $r_1,r_2>0$, $m\in\mathbb{N}$ and $s\in{0, ..., m-1}$ are fixed. Then: \\
 1) if $R>r_1+r_2$, $Z(r_1,r_2)={\O}$, $f\in C^{2m-2-s}(\mathbb{D}_R)$ and with $|z|<R-r$ holds (1), then $f\in C^{\infty}(\mathbb{D}_R)$ and satisfies (2);\\
 2) if $\max \{r_1,r_2\}<R<r_1+r_2$ and $Z(r_1,r_2)\neq{\O}$, then there is $f\in C^{\infty}(\mathbb{D}_R)$, that satisfies (1) with $|z|<R-r$ and does not satisfy (2).

\end{theorem}
As regards other two-radii theorems see papers \cite{1.} - \cite{9.} and references in these papers.

\section{Auxiliary Statements}
In this section we will obtain some auxiliary statements, that are necessary for the proof of main results. \\
First of all, we note that the function $g_r$ is an even entire function of exponential type, that grows as a polynomial on the real axis (see, for example, \cite{15.}, $\S$ 29). This together with the Hadamard theorem implies that the set $Z_r$ is infinite.

\begin{lemma}
 Let $\lambda\in Z_r$ and $|\lambda|>4/r$. Then
  \begin{equation}
   |\mathrm{Im}\lambda|\leq c_1\ln(1+|\lambda|),
   \end{equation}
  where constant $c_1$ is not depended on $\lambda$. \\
   Moreover, for all $\lambda$ with sufficiently large absolute value
   \begin{equation}
   |g_r^{\prime}(\lambda)|>\frac{c_2}{|\lambda|},
   \end{equation}
    where $c_2$ is not depended on $\lambda$. In addition, all zeros of the $g_r$ with sufficiently large absolute value are simple.
   \end{lemma}
  \begin{proof} [Proof]
   From the condition $g_r(\lambda)=0$ and asymptotic expansion for $J_{s+1}(\lambda r)$ as $\lambda\rightarrow\infty$ (see \cite{15.}, $\S$ 29) we have
  \[
  \sqrt{\frac{2}{\pi\lambda r}}\left(\cos (\lambda r-\frac{\pi s}{2}-\frac{3\pi}{4})-\frac{4(s^2+2s+1)-1}{8\lambda r}\sin(\lambda r-\frac{\pi s}{2}-\frac{3\pi}{4})\right)+
  \]
  \[
  +O\left((\lambda r)^{-2}e^{|\mathrm{Im}(\lambda r)|}\right)=(\lambda r)^{s+1} \overset{m-1}{\underset{n=s}{\sum}}\frac{(\lambda r)^{2n-2s}(-1)^{n-s-1}}{(2n+2)(n-s)!n!2^{2n-s}}.
    \]
  Hence, using $\lambda\in Z_r$, we obtain

    \[
    \frac{e^{i(\lambda r-\frac{\pi s}{2}-\frac{\pi}{4})}}{2i}+O\left(\frac{e^{|\mathrm{Im}(\lambda r)|}}{\lambda r}\right)=\sqrt{\frac{\pi\lambda r}{2}}\overset{m-1}{\underset{n=s}{\sum}}\frac{(\lambda r)^{2n-s+1}(-1)^{n-s-1}}{(2n+2)(n-s)!n!2^{2n-s}}.
    \]

Denote by $p_1(\lambda r)$ the polynomial from the right hand side of this equation. Then we have the following

      \[
      e^{i(\lambda r-\frac{\pi s}{2}-\frac{\pi}{4})}=2ip_1(\lambda r)+O\left(\frac{2ie^{|\mathrm{Im}(\lambda r)|}}{\lambda r}\right).
      \]

 Let us estimate
 \[e^{|\mathrm{Im}(\lambda r)|}\leq|2ip_1(\lambda r)|+\frac{|2i|e^{|\mathrm{Im}(\lambda r)|}}{\lambda r}\leq|2ip_1(\lambda r)|+\frac{|i|e^{|\mathrm{Im}(\lambda r)|}}{2}.\]

  Now one has
  \[e^{|\mathrm{Im}(\lambda r)|}\leq4|p_1(\lambda r)|\]
  and inequality  (7) is proved. Inequality (8) can be proved in a similar way, by using \cite{15.},  formula (6.3).
  \end{proof}

\begin{lemma}
  Let $\lambda\in\mathbb{C}$, $f(z)=e^{i\lambda(x\cos \alpha+y\sin\alpha)}$, $r>0$. Then for $z\in\mathbb{C}$ we have
 \[
 \underset{|\zeta-z|\leq{r}}{{\int}{\int}}f(\zeta)(\zeta-z)^sd\xi
d\eta - \overset{m-1}{\underset{n=s}{\sum}}\frac{2\pi r^{2n+2}}{2(n-s)!(n+1)!}{\left(\frac{\partial}{\partial{z}}\right)}^{n-s}{\left(\frac{\partial}{\partial{\bar{z}}}\right)}^nf(z)=
\]
\[
=2\pi g_r(\lambda)e^{i\alpha s}i^{s+2}\frac{r^{s+1}}{\lambda}e^{i\lambda(x\cos \alpha+y\sin\alpha)}
\]
\end{lemma}
\begin{proof} [Proof]
We substitute the function $e^{i\lambda(x\cos \alpha+y\sin\alpha)}$ to the right hand side of equation (1).

First, we have

\[
\underset{|w|\leq{r}}{{\iint}}f(w+z)w^sdudv=\underset{|w|\leq{r}}{{\iint}}e^{i\lambda ((x+u)\cos \alpha+(y+v)\sin\alpha)}w^sdudv=
\]
\[=e^{i\lambda(x\cos \alpha+y\sin\alpha)}\overset{\pi}{\underset{-\pi}{\int}}\overset{r}{\underset{0}{\int}}\left( \rho e^{i\varphi}\right)^se^{i\lambda\rho \cos (\varphi-\alpha)}\rho d\varphi d\rho.
\]

Let make the substitution $t=\varphi-\alpha$. Then

\[e^{i\lambda(x\cos \alpha+y\sin\alpha)}e^{i\alpha s}\overset{\pi}{\underset{-\pi}{\int}}\overset{r}{\underset{0}{\int}}\rho^{s+1}e^{its}e^{i\lambda \rho \cos  t}dtd\rho=
\]
\[=e^{i\lambda(x\cos  \alpha+y\sin\alpha)}e^{i\alpha s}\times\]
\[\times\overset{r}{\underset{0}{\int}}\rho^{s+1}(-1)\overset{\pi}{\underset{-\pi}{\int}}e^{-i(t+\frac{\pi}{2})s}e^{i\frac{\pi}{2}s}e^{i\lambda \rho \sin(\frac{\pi}{2}+t)}d\left( \frac{\pi}{2}+t\right)d\rho.\]

Continuing consideration, we obtain

\[e^{i\lambda(x\cos  \alpha+y\sin\alpha)}e^{i\alpha s}\overset{\pi}{\underset{-\pi}{\int}}\overset{r}{\underset{0}{\int}}\rho^{s+1}e^{its}e^{i\lambda \rho \cos  t}dtd\rho=\]
\[=e^{i\lambda(x\cos \alpha+y\sin\alpha)}e^{i\alpha s}i^s2\pi (-1)\overset{r}{\underset{0}{\int}}\rho^{s+1}J_s(\lambda\rho)d\rho.\]

Now from the properties of the Bessel function  $J_s(z)$ we deduce the next

\[e^{i\lambda(x\cos \alpha+y\sin\alpha)}e^{i\alpha s}i^s(-2\pi)\frac{1}{\lambda^{s+2}}\overset{r}{\underset{0}{\int}}(\lambda\rho)^{s+1}J_s(\lambda\rho)d(\lambda\rho)=\]
\[=e^{i\lambda(x\cos \alpha+y\sin\alpha)}e^{i\alpha s}i^s\frac{(-2\pi)}{\lambda}r^{s+1}J_{s+1}(\lambda).\]

Then we substitute this function to the left hand side of our equation.

\[2\pi\overset{m-1}{\underset{n=s}{\sum}}\frac{r^{2n+2}}{(2n+2)(n-s)!n!}
{\left(\frac{\partial}{\partial{z}}\right)}^{n-s}{\left(\frac{\partial}{\partial{\bar{z}}}\right)}^n\left (e^{i\lambda(x\cos \alpha+y\sin\alpha)}\right)=\]
\[=2\pi\overset{m-1}{\underset{n=s}{\sum}}\frac{r^{2n+2}}{(2n+2)(n-s)!n!}
\frac{i^{2n-s}}{2^{2n-s}}\lambda^{2n-s}e^{i\alpha s}e^{i\lambda(x\cos \alpha+y\sin\alpha)}.\]

It is clear that the difference of obtained expressions for the right and left hand sides has the form  $2\pi g_r(\lambda)e^{i\alpha s}i^{s+2}\frac{r^{s+1}}{\lambda}e^{i\lambda(x\cos \alpha+y\sin\alpha)}$.
\end{proof}
\begin{corollary}
Let $\lambda\in Z_r$, $\eta\in \{0,...,n_\lambda-1\}$, $\alpha\in R^1$.
Then the function
\[
 {\left(\frac{\partial}{\partial{z}}\right)}^\eta e^{i\lambda(x\cos \alpha+y\sin\alpha)}
 \]
  satisfies (1) for all $z\in\mathbb{C}$. \\
  The same statement is true for the function $\Phi_{\lambda,\eta,k}(\rho)e^{ik\varphi}$ with any $k\in\mathbb{Z}$.
\end{corollary}
\begin{proof}[Proof]
The proof follows from the Lemma 2 and [5,  formula (1.5.29)].
\end{proof}
\begin{lemma}
 Let $m\in\mathbb{N}$ and $s\in\{0,..,m-1\}$ are fixed. Then $f\in C^{2m-s}(\mathbb{D}_R)$ satisfies (2) if and only if for all $k\in\mathbb{Z}$ and $\rho\in[0,R)$ the next equality is true
\begin{equation}
f_k(\rho)=\underset{p+k\geq0}{\underset{0\leq p\leq s-1}{\sum}}a_{k,p}\rho^{2p+k}+\overset{m-s-1}{\underset{p=0}{\sum}}b_{k,p}\rho^{2p+s+|k+s|},
\end{equation}
where $a_{k,p}\in \mathbb{C}$ and $b_{k,p}\in \mathbb{C}$.
\end{lemma}
\begin{proof}[Proof]
 In the case where $b_{k,p}=0$ and equality ${\left(\frac{\partial}{\partial{\bar{z}}}\right)}^mf=0$ is considered instead of a similar statement was proved in \cite{10.}. In our case the proof is carried out by the same lines.
\end{proof}

\begin{lemma}
  Let $m\in\mathbb{N}$ and $s\in\{0,..,m-1\}$ are fixed. Assume that a function $f\in C^{\infty}(\mathbb{D}_R)$ satisfies (1) with fixed $r<R$ and all $z\in \mathbb{D}_{R-r}$. \\
 Let $f=0$ in $\mathbb{D}_r$. Then $f\equiv0$.
\end{lemma}

\begin{proof}[Proof]
The statement of Lemma 4 is a special case Theorem 1 from \cite{16.}.
\end{proof}
\section{Proof of Theorem 1}
{Sufficiency}. First, let $f\in C^{\infty}(\mathbb{D}_R)$ and   equality (5) holds on $[0,R)$ for any $k\in \mathbb{Z}$ with the coefficients, that satisfy (6). From Lema 2 and Corollary 1 we see, that function  $f_k(\rho)e^{ik\varphi}$ satisfies (1) with $|z|<R-r$. Because of the arbitrariness of $k\in \mathbb{Z}$ this together with (3), (4) implies (see, for example, [5, Section 1.5.2] that the function $f$ also satisfies (1) with $|z|<R-r$. Hence, implication 2)$\rightarrow$ 1) is proved. \\
Now we prove the reverse statement. \\
Let $\mathcal{E}_{\natural}^{\prime}(\mathbb{C})$ denote the space of radial compactly supported distributions on $\mathbb{C}$. Let $f\in C^{\infty}(\mathbb{D}_R)$ and assume that  equality (1) holds for $|z|<R-r$. From [5, statement 1.5.6] the functions $F_k(z)=f_k(\rho)e^{ik\varphi}$ satisfy this condition as well. Using the Paley-Wiener  theorem for the spherical transform (see [5, Section 3.2.1 and Theorem 1.6.5]), we define the distribution $T\in \mathcal{E}_{\natural}^{\prime}(\mathbb{C})$ with support in $\overline{\mathbb{D}}_r$ by the following formula
\[
\widetilde{T}(z)=g_r(z), z\in \mathbb{C}.
\]
 A calculation shows that equality (1) holds for the function $F_k$ with $|z|<R-r$. This is equivalent to the following convolution equation

$F_k\ast{\left(\frac{\partial}{\partial z}\right)}^{m-s}{\left(\frac{\partial}{\partial{\bar{z}}}\right)}^mT=0$ in $\mathbb{D}_{R-r}$.

We solve this equation by using Lemma 1 - 4. Then we have (see [5, Section 3.2.4]) statement 2). \\
Hence the theorem.
\section{Proof of Theorem 2}

  Let $R>r_1+r_2$, $Z(r_1,r_2)={\O}$, $f\in C^{2m-2-s}(\mathbb{D}_R)$ and assume that equality (1) holds for $|z|<R-r$. Let us prove that $f$ satisfies (2) in $\mathbb{D}_R$. \\
 Without loss of the generality, we can suppose that $f\in C^{\infty}(\mathbb{D}_R)$ (the general case can be reduces to this one by the standard smoothing, see [5, Section 1.3.3]). \\
 By Theorem 1, for any $k\in \mathbb{Z}$ and $\rho\in [0, R)$ the next equality holds
\[
f_k(\rho)e^{ik\varphi}= \underset{p+k\geq0}{\underset{0\leq p\leq s-1}{\sum}}a_{k,p}\rho^{2p+k}e^{ik\varphi}+\overset{m-s-1}{\underset{p=0}{\sum}}b_{k,p}\rho^{2p+s+|k+s|}e^{ik\varphi}+
\]
\begin{equation}
+\underset{\lambda\in{Z_{r_1}}}{\sum}\overset{n_\lambda-1}{\underset{\eta=0}{\sum}}
c_{\lambda,\eta,k}\Phi_{\lambda,\eta,k}(\rho)e^{ik\varphi},
\end{equation}
 where $a_{k,p}\in\mathbb{C}$, $b_{k,p}\in\mathbb{C}$ and the constants $c_{\lambda,\eta,k}$ satisfy (6).\\
From this condition it follows that the series in (10) converges in the space $C^{\infty}(\mathbb{D}_R)$ (see [5, Lemma 3.2.7]). \\
Let
\[
F_k(z)={\left(\frac{\partial}{\partial{z}}\right)}^{m-s}{\left(\frac{\partial}{\partial{\bar{z}}}\right)}^m\left(f_k(\rho)e^{ik\varphi}\right)=
\]
\begin{equation}
=\underset{\lambda\in{Z_{r_1}}}{\sum}\overset{n_\lambda-1}{\underset{\eta=0}{\sum}}
c_{\lambda,\eta,k}{\left(\frac{\partial}{\partial{z}}\right)}^{m-s}{\left(\frac{\partial}{\partial{\bar{z}}}\right)}^m\Phi_{\lambda,\eta,k}(\rho)e^{ik\varphi}.
\end{equation}
  In view of (11) we see that $F_k\ast T_1=0$ in $\mathbb{D}_{R-r_1}$, where the distribution $T_1\in \mathcal{E}^{\prime}_{\natural}(\mathbb{C})$ with support in $\overline{\mathbb{D}}_{r_1}$ is determined by the equality $\widetilde{T}_1(z)=g_{r_1}(z)$ (see [5, Theorem 1.6.5]). \\
Similarly, using Theorem 1 for $r=r_2$, we conclude that $F_k\ast T_2=0$ in $\mathbb{D}_{R-r_2}$, where $T_2\in \mathcal{E}^{\prime}_{\natural}(\mathbb{C})$ with support in $\overline{\mathbb{D}}_{r_2}$ is determined by the equality $\widetilde{T}_2(z)=g_{r_2}(z)$.\\
 If $Z(r_1,r_2)={\O}$ then from [5, Theorem 3.4.1] we conclude that $F_k=0$. \\
  Then it follows from (11) that the function $f_k(\rho)e^{ik\varphi}$ satisfies (2) for all $k\in \mathbb{Z}$. It means  that (see [5, proof of the Lemma 2.1.4]) $f$ satisfies (2). Thus the first statement of Theorem 2 is proved. \\
 We now establish the second statement. \\
  If there is $\lambda\in Z(r_1,r_2)$ then the function $f(z)=\Phi_{\lambda, 0, 0}(|z|)$ does not satisfy (2). In addition, it satisfies (1) for all $z\in \mathbb{C}$ and $r=r_1,r_2$ (see Corollary 1). Then we henceforth assume that $Z(r_1,r_2)={\O}$.\\
    Suppose that $T_1,T_2\in \mathcal{E}^{\prime}_{\natural}(\mathbb{C})$ are defined as above. If $R<r_1+r_2$, in view of  [5, Theorem 3.4.9] we conclude, that there is a nonzero radial function $f\in C^{\infty}(\mathbb{D}_R)$. It satisfies the conditions $f\ast T_1=0$ in $\mathbb{D}_{R-r_1}$ and $f\ast T_2=0$ in $\mathbb{D}_{R-r_2}$. \\
 Applying [5, Theorem 3.2.3] we infer that for $r=r_1,r_2$ the following equality holds
 \[
 f(z)=\underset{\lambda\in{Z_r}}{\sum}\overset{n_\lambda-1}{\underset{\eta=0}{\sum}}c_{\lambda,\eta}(r)\Phi_{\lambda,\eta,0}(|z|), \]
where $z\in \mathbb{D}_R$ and the constants $c_{\lambda,\eta}(r)$ satisfy (5). Moreover, these constants are not all equal to zero.\\
From this equality and Corollary 1 one deduces that $f$ satisfies (1) for $|z|<R-r$, $r=r_1,r_2$. \\
Suppose now that $f$ satisfies (2). \\
 Then $f(z)=\underset{0\leq p\leq s-1}{\sum}a_p|z|^{2p}+\overset{m-s-1}{\underset{p=0}{\sum}}b_p|z|^{2p+2s}$ in $\mathbb{D}_R$ and the convolutions $f\ast T_1$ and $f\ast T_2$ are polynomials. This means that $f\ast T_1=f\ast T_2=0$ in $\mathbb{C}$.\\
  Since $Z(r_1,r_2)={\O}$, from [5, Theorem 3.4.1] we infer that $f=0$. This contradicts by the definition of $f$. \\
Therefore, the function $f$ satisfies all the requirements of the second statement of Theorem2.


{\footnotesize

\vsk

Faculty of Mathematics and Information Technology,
 Donetsk National University\\
  Donetsk, Ukraine, Universitetskaya 24, 83001\\
odtrofimenko@gmail.com

 }


\begin{thebibliography}{9}
\parskip-2pt
   \bibitem{1.} L.Zalcman. A bibliographic survey of the Pompeiu problem, in: B.Fuglede et al. (ads.) \sep Approximation by Solutions of Partial Differential Equations, Kluwer Academic Publishers: Dordrecht. -- 1992. -- pp.185-194.

  \bibitem{2.} L.Zalcman. Supplementary bibliography to 'A bibliographic survey of the Pompeiu problem \sep Radon Transforms and Tomography, Contemp. Math., {\bf 278} -- 2001. -- p.69-74.

  \bibitem{3.} L.Zalcman. Mean values and differential equations \sep Israel J. Math., {\bf14} -- 1973. -- pp.339-352.

  \bibitem{4.}  L.Zalcman. Offbeat integral geometry \sep Amer. Math. Monthly, {\bf87}, 3 -- 1980. -- pp.161-175.

  \bibitem{5.} Volchkov V.V. Integral Geometry and Convolution Equation.
Dordrecht-Boston-London: Kluwer Academic Publishers, 2003. 454p.

 \bibitem{6.} Volchkov V.V., Volchkov Vit.V. Harmonic Analysis of Mean Periodic Functions on Symmetric spaces and the Heisenberg Group. Series: Springer Monographs in Mathematics, 2009. – 671 p.

 \bibitem{7.} C.A.Berenstein and D.C.Struppa. Complex analysis and convolution equations, in 'Several Complex Variables, V', (G.M.Henkin, Ed.) \sep Encyclopedia of Math. Sciences, {\bf54} -- 1993. -- pp.1-108.

 \bibitem{8.} I.Netuka and J.Vesely. Mean value property and harmonic functions \sep Classical and Modern Potential Theory and Applications, Kluwer Acad. Publ. -- 1994. -- pp.359-398.

 \bibitem{9.} J.Delsarte. Lectures on Topics in Mean Periodic Functions and the Two-Radius Theorem. Tata Institute: Bombay. -- 1961.

 \bibitem{10.} Trofymenko O.D. Generalization of the mean value theorem for polyanalytic functions in the case of a circle and a disk \sep Bulletin of Donetsk University -- 1,
   2009.  --  pp.28-32 (in Ukrainian).

 \bibitem{11.} L.Zalcman. Analyticity and the Pompeiu problem \sep Arch. Rat. Anal. Mech., {\bf47} -- 1972. -- pp.237-254.

 \bibitem{12.} Volchkov V.V. New mean value theorems for polyanalytic functions \sep Mat. Zametki, {\bf56}, 3 -- 1994. -- pp.20-28.

 \bibitem{13.} Maxwell O.Reade. A theorem of Fedoroff \sep Duke Math.J,
   {\bf18}. -- 1951. --  pp.105-109.

  \bibitem{14.} T.Ramsey and Y.Weit. Mean values and classes of harmonic functions \sep Math. Proc. Camb. Dhil. Soc., {\bf96}, 1984. -- pp.501-505.

  \bibitem{15.} Korenev B.G. Introduction to the theory of Bessel functions. M.: Nauka, 1971. -  288 p.

  \bibitem{16.}  Trofymenko O.D. Uniqueness theorem for solutions of some mean value equations \sep Donetsk: Transactions of the Institute of Applied Mathematics and Mechanics, {\bf24}, 2012. -- pp.234-242.

\end{thebibliography}
\end{document}